\documentclass[12pt]{article}
\usepackage[english]{babel}
\usepackage{amsmath,amsthm}
\usepackage{amsfonts}
\usepackage{graphicx}
\usepackage{enumerate}
\usepackage[usenames,dvipsnames]{color}
\usepackage{subfigure}
\usepackage[right,pagewise,displaymath, mathlines]{lineno}
\usepackage{epstopdf}
\usepackage{color}
\usepackage{multirow}

\numberwithin{equation}{section}
\numberwithin{figure}{section}
\numberwithin{table}{section}

\newtheorem{theorem}{Theorem}[section]

\newtheorem{lemma}{Lemma}[section]

\newtheorem{note}{Note}[section]

\numberwithin{equation}{section}

\begin{document}

\begin{center}

{\Large\bf Beyond the Pearson correlation: heavy-tailed risks, weighted Gini correlations, and a Gini-type weighted insurance pricing model}

\vspace*{7mm}

{\large Edward Furman\footnote{Corresponding author.
Email: efurman@mathstat.yorku.ca
}}

\medskip

\textit{Department of Mathematics and Statistics, York University,\\
Toronto, Ontario M3J 1P3, Canada}

\bigskip

{\large Ri\v cardas Zitikis}

\medskip

\textit{Department of Statistical and Actuarial Sciences, University of
Western Ontario,\\ London, Ontario N6A 5B7, Canada}

\end{center}

\medskip

\begin{quote}
\textbf{Abstract.} Gini-type correlation coefficients have become increasingly important in a variety of research areas, including economics, insurance and finance, where modelling with heavy-tailed distributions is of pivotal importance. In such situations, naturally, the classical Pearson correlation coefficient is of little use. On the other hand, it has been observed that when light-tailed situations are of interest, and hence when both the Gini-type and Pearson correlation coefficients are well-defined and finite, then these coefficients are related and sometimes even coincide. In general, understanding how the correlation coefficients above are related has been an illusive task. In this paper we put forward arguments that establish such a connection via certain regression-type equations. This, in turn, allows us to introduce a Gini-type Weighted Insurance Pricing Model that works in heavy-tailed situation and thus provides a natural alternative to the classical Capital Asset Pricing Model. We illustrate our theoretical considerations using several bivariate distributions, such as elliptical and those with heavy-tailed Pareto margins.

\medskip


\medskip


\end{quote}

\newpage

\section{Introduction}
\label{sec-intro}

The topic of measuring association has been of profound interest in many theoretical and applied research areas. The Pearson correlation coefficient, $\rho:=\rho[X, Y]$ for two random variables (r.v.'s) $X$ and $Y$  has arguably been one of the most popular members of every researcher's toolbox, notwithstanding its known shortcomings, and especially the requirement of finite second moments. The latter requirement can be a real impediment because modelling in economics, finance, and insurance frequently requires heavy-tailed distributions. Furthermore, the coefficient $\rho  $ can turn out to be very small, and even zero, when the r.v.'s $X$ and $Y$ are  strongly dependent. This should not, of course, be surprising because $\rho$ aggregates, in the form of an integral, the values $F_{X,Y}(x,y)-F_X(x)F_Y(y)$ over all real $x $ and $ y$, where $F_{X,Y}$ denotes the joint cumulative distribution function (c.d.f.) of $X$ and $Y$, and  $F_X$ and $F_Y$ are the respective marginal c.d.f.'s. The desire to avoid some of these, and other, limitations of the Pearson correlation coefficient has naturally grown into a strong impetus for seeking alternatives (cf., e.g., Schechtman and Yitzhaki, 2013).

The monograph by Schechtman and Yitzhaki (2013) makes a number of convincing arguments, supported by illustrative examples, in favour of the extended Gini correlation coefficient
\begin{equation}
\label{Gini-cor-0}
\Gamma_\gamma[X, Y]
=\frac{\mathbf{Cov}[X,(1-F_Y(Y))^{\gamma}]}{\mathbf{Cov}[X, (1-F_X(X))^{\gamma}]}, 
\end{equation}
indexed by parameter $\gamma> 0$. When $\gamma =1$, then the extended Gini correlation coefficient reduces to the classical Gini correlation coefficient, succinctly 
\[
\Gamma[X, Y]= \frac{\mathbf{Cov}[X,F_Y(Y)]}{\mathbf{Cov}[X, F_X(X)]}.
\]
Various properties of  $\Gamma_\gamma[X, Y]$ have been documented and discussed in the aforementioned monograph (see also 
Samanthi et al., 2016). For example, we learn that when the pair $(X, Y)$ follows the bivariate normal distribution, then $\Gamma_\gamma[X, Y]$ coincides with the Pearson correlation coefficient $\rho[X, Y]$ irrespectively of the value of $\gamma> 0$. This property, obviously, does not hold in general, which gives rise to a natural question as to when the two coefficients are equal, but when they are not, then how much they differ. These are among the problems that we tackle in this paper.

On a different note,
the power weight function $w(t)=t^{\gamma }$ that transforms the de-cumulative distribution functions (d.d.f.'s) $\overline{F}_X(x):=1-F_X(x)$ and $\overline{F}_Y(y):=1-F_Y(y)$ in the definition of $\Gamma_\gamma[X, Y]$ might not be sufficiently flexible to adequately emphasize (or de-emphasize) certain portions of the ranks of $X$ and $Y$. For example, in prospect theory (e.g., Tversky and Kahneman, 1992; Sriboonchita at al., 2009; Wakker, 2010; and references therein) we find arguments in favour of sigmoidal and more complexly-shaped weight functions. To accommodate various weighing designs, we therefore suggest using the weighted Gini correlation coefficient
\begin{equation}\label{Cor-w}
C_w[X,Y]=\frac{\mathbf{Cov}[X,w(1-F_Y(Y))]}{\mathbf{Cov}[X, w(1-F_X(X))]},
\end{equation}
where the weight function $w:[0, 1]\rightarrow [0, 1]$ can be any non-decreasing function for which the numerator and the denominator are well-defined and finite. The methodology that we develop in this paper can successfully tackle this correlation coefficient for large classes of weight functions and bivariate distributions.

We have organized the rest of this paper as follows. In Section \ref{sec-MainRes} we present and discuss general properties of the weighted Gini correlation coefficient, as well as its relationships to the Capital Asset Pricing Model (CAPM) and Weighted Insurance Pricing Model (WIPM). In Section \ref{sub-sec-comp}, we present a technique that relates the weighted Gini correlation coefficient to that of Pearson. Since the latter is closely linked to the bivariate normal distribution, in the same section we discuss the aforementioned link in the case of bivariate normal and, more generally, elliptical distributions. In Sections \ref{sec-Par-Arnold}, \ref{sec-Par-SF1} and \ref{sec-Par-SF2} we deal, in increasing complexity, with three bivariate Pareto-type distributions that model pairs $(X,Y)$ with i) exchangeable margins and linear regression, ii) non-exchangeable margins and linear regression, and iii) non-exchangeable margins and non-linear regression. Proofs are relegated to Appendix \ref{appendix}. To somewhat simplify the presentation, throughout the paper we use $\mathcal{W}$ to denote the class of non-decreasing (weight) functions $w:[0, 1]\rightarrow [0, 1]$.

\section{Properties of weighted Gini correlations}
\label{sec-MainRes}

We have already noted several drawbacks of the Pearson correlation coefficient $\rho[X, Y]$ when dealing with financial and insurance data, especially when the underlying r.v.'s do not have finite second moments. Some other well-known drawbacks of relevance to our present work are:
\begin{itemize}
\item The infinum and supremum of $\rho[X, Y]$ over all copulas governing the dependence between $X$ and $Y$ generally depend on the marginal c.d.f.'s $F_X$ and $F_Y$ (Shih and Huang, 1992). Specifically, given an arbitrary pair of marginal c.d.f.'s, we cannot claim that there is a pair $(X,Y)$ whose Pearson correlation coefficient is $+1$, but there is such a pair (with its dependence described by the co-monotonic copula) in the case of the Gini correlation coefficient.
\item The Pearson correlation coefficient $\rho[X, Y]$ is not invariant under increasing and non-linear transformations of the r.v.'s $X$ and $Y$
(Kendall and Stuart, 1979). In the case of the Gini correlation coefficient we have that $\Gamma[X,h(Y)]=\Gamma[X,Y]$ for any increasing function $h$.
\end{itemize}

Among the advantages of the Pearson correlation coefficient are its intuitive appeal and  computing easiness. The coefficient manifests naturally and plays a pivotal role in a myriad of situations. For example, it is a parameter of the bivariate normal distribution, and it measures the non-diversifiable risk in the Capital Asset Pricing Model (CAPM) (Sharpe, 1964; Lintner, 1965;
Black, 1972) and in its insurance counterpart (Furman and Zitikis, 2009) called the Weighted Insurance Pricing Model (WIPM).

We next discuss properties of the weighted Gini correlation coefficient and in this way highlight its superiority (at least in insurance and financial contexts) when compared to the Pearson correlation coefficient.

\begin{theorem}[Normalization]\label{th-1}
Let the c.d.f.'s $F_X$ and $F_Y$ be continuous functions. Then 
\begin{equation}\label{c-1}
-\lambda_w[X] \le C_w[X,Y]\le 1
\end{equation}
for every $w\in \mathcal{W}$, where
\[
\lambda_w[X]=\frac{\mathbf{Cov}[X,w(F_X(X))]}{\mathbf{Cov}[X, w^*(F_X(X))]}
\]
with $w^*(t)=1-w(1-t)$. Furthermore,
\begin{enumerate}[\rm(a)]
\item
the upper bound is achieved with $Y=X$, in which case we have $C_w[X,X]= 1$;
\item
the lower bound is achieved with $Y=-X$, in which case $C_w[X,-X]= -\lambda_w[X]$;
\item when $w(t)=t$ (the classical Gini correlation case), then $\lambda_w[X]=1$.
\end{enumerate}
\end{theorem}

\begin{theorem}[Sign]\label{th-2}
For every $w\in \mathcal{W}$, we have
\begin{enumerate}[\rm(a)]
\item\label{part-a}
$C_w[X,Y]=0$ when $X$ and $Y$ are independent;
\item\label{part-b}
$C_w[X,Y]\ge 0$ when $X$ and $Y$ are positively quadrant dependent (PQD);
\item\label{part-c}
$C_w[X,Y]\le 0$ when $X$ and $Y$ are negatively quadrant dependent (NQD).
\end{enumerate}
\end{theorem}

\begin{theorem}[Symmetry]\label{th-3}
Suppose there are real constants $a,c\in \mathbf{R}$ and positive constants $b,d>0$, such that the random variables $X_{a,b}=a+bX$ and $Y_{c,d}=c+dY$ are interchangeable, that is, the distribution of $(X_{a,b},Y_{c,d})$ is equal to that of $(Y_{c,d},X_{a,b})$.
Then for every $w\in \mathcal{W}$ we have
\begin{equation}\label{c-4}
C_w[X,Y]=C_w[Y,X].
\end{equation}
\end{theorem}

\begin{theorem}[Invariance under incrreasing transformations]\label{th-4}
Let $a\in \mathbf{R}$ and $b>0$ be some constants, and let $h:[0, 1]\rightarrow \mathbf{R}$ be an increasing function. Then for every $w\in \mathcal{W}$ we have
\begin{equation}\label{c-5}
C_w[a+bX,h(Y)]=C_w[X,Y].
\end{equation}
Hence, in particular, $C_w[a+bX,c+dY]=C_w[X,Y]$ for all real constants $a,c\in \mathbf{R}$ and all positive constants $b,d>0$.
\end{theorem}

Before stating our next theorem, we recall that for the set $\mathcal{X}$ of insurance risks $X$ and non-decreasing weight functions
$v:[0, \infty)\rightarrow[0, \infty)$, the functional
$\Pi_v:\mathcal{X}\times\mathcal{X}\rightarrow [0, \infty]$ defined by the equation 
\begin{equation}
\label{eq-Ewpcp}
\Pi_v[X,Y]=\frac{\mathbf{E}[Xv(Y)]}{\mathbf{E}[v(Y)]}
\end{equation}
is called the economic weighted premium calculation principle (p.c.p.),
whereas the equation (Furman and Zitikis, 2009)
\begin{equation}
\label{eq-WIPM}
\Pi_v[X, Y]=\mathbf{E}[X]+\rho[X, Y]
\sqrt{\frac{\mathbf{Var}[X]}{\mathbf{Var}[Y]}}
\left(
\pi_v[Y]-\mathbf{E}[Y]
\right), 
\end{equation}
which holds with $\pi_v[Y]:=\Pi_v[Y, Y]$ under certain assumptions on the joint c.d.f.\, of $(X,Y)$, is  referred to as the Weighted Insurance Pricing Model (WIPM). For example, two natural choices of the r.v.\, $Y$ would be i) leverage risk, which is of interest to the pricing actuary, and ii) the aggregate risk $S=X+Y$. In the latter case, the WIPM equation can be employed as an economic capital allocation rule, and we refer to Furman and Zitikis (2009) for details on this topic.

An obvious disadvantage of the WIPM -- akin to the CAPM -- is its reliance on the assumption that the variances of underlying r.v.'s are finite. Consequently, the WIPM cannot be applied to a multitude of risks. To circumvent the problem, in definition (\ref{eq-Ewpcp}) we set $v(y)=w\circ (1-F_Y(y))$, which gives rise to the following Gini economic p.c.p.
\begin{equation}
\label{eq-Ewgpcp}
\Pi_{G,w}[X, Y]=\frac{\mathbf{E}[Xw(1-F_Y(Y))]}{\mathbf{E}[w(1-F_Y(Y))]}.
\end{equation}

\begin{theorem}[Connection to WIPM]\label{th-5}
Let $(X, Y)$ be a pair of r.v.'s in $\mathcal{X}$.
When there exist constants $\alpha,\beta\in\mathbf{R}$ such that $\mathbf{E}[X|\ Y]=\alpha+\beta Y$, then the Gini WIPM counterpart to equation (\ref{eq-WIPM}) is
\begin{equation}
\Pi_{G,w}[X, Y]=\mathbf{E}[X]+C_w[X, Y]
\frac{\mathbf{Cov}[X,w(1-F_X(X))]}{\mathbf{Cov}[Y, w(1-F_Y(Y))]}
\left(
\pi_{G,w}[Y]-\mathbf{E}[Y]
\right), 
\label{eq-WIPM-gini}
\end{equation}
where $\pi_{G,w}[Y]:=\Pi_{G,w}[Y, Y]$.
\end{theorem}

Unlike the WIPM equation (\ref{eq-WIPM}), Gini WIPM equation (\ref{eq-WIPM-gini}) does not need finiteness of the second moments of $X$ and $Y$, and as such it can be used for pricing/measuring risks with infinite variances. Also, Theorem \ref{th-5} shows that while the Pearson correlation coefficient arises in the context of the CAPM, the weighted Gini correlation coefficient $C_w[X, Y]$ manifests in the context of the Gini variant of WIPM. Finally, Theorem \ref{th-5} provides yet another justification for the exploration of computational tractability of the class of Gini correlations, with which we deal throughout the rest of this paper.

\section{The Gini and Pearson correlations connected}
\label{sub-sec-comp}

At first sight, the weighted Gini correlation coefficient $C_w[X,Y]$ and the Pearson correlation coefficient $\rho[X, Y]$ seem to be quite different: the former aggregates the value of $X$ with the ranks of $Y$, and also with the ranks of $X$, whereas the Pearson correlation coefficient couples the values of $X$ and $Y$ and does not rely on the ranks of any of them. Nevertheless, the two coefficients are related as the following theorem shows.

\begin{theorem}\label{th-el-1}
When there exist constants $\alpha,\beta\in\mathbf{R}$ such that $\mathbf{E}[X|\ Y]=\alpha+\beta Y$, then the equation
\begin{equation}
\label{eq-0c}
C_w[X, Y]=\beta\frac{\mathbf{Cov}[Y, w(1-F_Y(Y))]}{\mathbf{Cov}[X, w(1-F_X(X))]}
\end{equation}
holds for every $w\in \mathcal{W}$. The converse is also true, assuming continuity of the c.d.f.\, $F_Y$.
\end{theorem}

The most important part of Theorem \ref{th-el-1} is of course the fact that having a linear regression function allows us to decompose the weighted (and thus  extended and classical) Gini correlation coefficient into i) covariance-type functionals based on the marginal c.d.f.'s $F_X$ and $F_Y$, and ii) the slope $\beta $ of the regression line whose immense role in the classical Capital Asset Pricing Model (CAPM) has been well explored and documented (e.g., Levy, 2011; and references therein). The main purpose of the converse of the theorem, which we establish under an additional but very mild condition, is to show that the main part cannot be improved in general, that is, the linearity of the regression function is pivotal for the form of $C_w[X, Y]$ spelled out on the right-hand side of equation (\ref{eq-0c}).

We now illustrate Theorem \ref{th-el-1} when the pair $(X, Y)$ follows the bivariate normal distribution $N_2(\boldsymbol{\mu}, \Sigma)$ with the vector $\boldsymbol{\mu}=(\mu_X,\mu_Y)$ of the marginal expectations and the symmetric matrix $\Sigma$ whose main diagonal is made up of the variances $\sigma_X^2=\mathbf{Var}[X]$ and
$\sigma_Y^2=\mathbf{Var}[Y]$, and the two off-diagonal entries are equal to the covariance $\sigma_{X,Y}=\mathbf{Cov}[X,Y]$. It is well known that the equation $\mathbf{E}[X|\ Y]=\alpha+\beta Y$ holds with the intercept $\alpha = \mu_X-\beta\mu_Y$ and the slope $\beta=\sigma_{X,Y}/\sigma_Y^2$.
Hence, an application of Theorem \ref{th-el-1} establishes the equation
\begin{equation}
\label{Cw-e0}
C_w[X, Y]=\rho[X, Y]
\end{equation}
for every $w\in \mathcal{W}$. This equation implies the well-known fact (e.g., Schechtman and Yitzhaki, 2013) that the Pearson and the extended Gini correlation coefficients are equal in the bivariate normal case. Hence, if the random pair $(X, Y)$ follows the bivariate normal distribution, then the Gini-type WIPM equation reduces to 
\begin{equation}
\Pi_{G,w}[X, Y]=\mathbf{E}[X]+\rho[X,Y]
\sqrt{\frac{\mathbf{Var}[X]}{\mathbf{Var}[Y]}}
\left(
\pi_{G,w}[Y]-\mathbf{E}[Y]
\right).
\label{eq-WIPM-gini-norm}
\end{equation}

More generally, let the pair $(X, Y)$ follow the bivariate elliptical distribution $E_2(\boldsymbol{\mu}, \Sigma)$ with the vector $\boldsymbol{\mu}=(\mu_X,\mu_Y)$ of marginal expectations and the positive-definite matrix $\Sigma$ whose diagonal entries are $\sigma_X^2$ and $\sigma_Y^2$ and the two identical off-diagonal entries are $\sigma_{X,Y}$. We note that, unlike in the bivariate normal case, the just introduced $\sigma$'s may or may not be the variances and covariances of $X$ and $Y$. Nevertheless, we have (e.g., Fang et al., 1990) the equation $\mathbf{E}[X|\ Y]=\alpha+\beta Y$ with the intercept $\alpha = \mu_X-\beta\mu_Y$ and the slope $\beta=\sigma_{X,Y}/\sigma_Y^2$, and so Theorem \ref{th-el-1} applies and gives the equation
\begin{equation}
\label{Cw-el}
C_w[X, Y]=\frac{\sigma_{X,Y}}{\sigma_X\sigma_Y}
\end{equation}
for every $w\in \mathcal{W}$. Consequently, just like in the bivariate normal case, the Pearson and extended Gini correlation coefficients coincide whenever the former exists. When it does not exist, which says that we are dealing with heavy-tailed random variables, then $\Gamma_\gamma[X, Y]$ and likewise $C_w[X, Y]$ can be viewed as  extensions of the Pearson correlation coefficient to those pairs of $X$ and $Y$ that are outside the domain of definition of $\rho[X, Y]$. Finally, 
the Gini-type WIPM equation is given in this case by

\begin{equation}
\Pi_{G,w}[X, Y]=\mathbf{E}[X]+\frac{\sigma_{X,Y}}{\sigma_Y^2}
\left(
\pi_{G,w}[Y]-\mathbf{E}[Y]
\right).
\label{eq-WIPM-gini-ell}
\end{equation}

\section{Exchangeable linearly-regressed margins}
\label{sec-Par-Arnold}

The elliptical distribution is a natural extension of the normal one and has served as an adequate model in many financial problems. Its role in Economics has not been particularly pronounced, where the classical Pareto distribution with its numerous extensions and variations have dominated the scene. Many insurance risk models have also relied on extensions and  generalizations of the Pareto distribution (e.g., Brazauskas and Serfling, 2003; Goovaerts et al., 2005). Because of this reason, and to emphasize that Theorem \ref{th-el-1} does not hinge on the symmetry of the joint distribution of $X$ and $Y$, we devote the rest of this paper to several bivariate Pareto-type distributions in the context of Theorem \ref{th-el-1}.

We begin with the univariate Pareto distribution of the 2nd kind, whose  d.d.f.\, is
\[
\overline{F}(x)=\left(1+\frac{x-\mu}{\sigma}\right)^{-\delta}
\]
for all $x>\mu$, where $\mu\in\mathbf{R}$ and $\sigma>0$ are location and scale parameters, respectively, and $\delta>0$ is tail index. The
distribution has been a classical example of heavy-tailness:
when $\delta\in (1, 2]$, then the variance does not exist, and
when $\delta\in (0, 1]$, then even the expectation does not exist.

When it comes to bivariate extensions of this distribution, there are many of them to consider. In this section we concentrate on the bivariate Pareto distribution of the 2nd type (Arnold, 1983) whose joint d.d.f.\, is
\begin{equation}
\label{ddf-Arnold}
\overline{F}_{X,Y}(x,y)
=\left(1+\frac{x-\mu_X}{\sigma_X}
+\frac{y-\mu_Y}{\sigma_Y}\right)^{-\delta}
\end{equation}
for all $x>\mu_X $ and $ y>\mu_Y$, where $\mu_X,\mu_Y\in\mathbf{R}$ are location parameters, $\sigma_X,\sigma_Y>0$ are scale parameters, and $\delta>0$ is tail index. (Despite notational similarities, the two $\mu$'s are not the means, and the two $\sigma$'s are not the standard deviations of $X$ and $Y$.) In Figure \ref{Fig-Arnold}
\begin{figure}[h!]
\centering
\includegraphics[width=12cm]{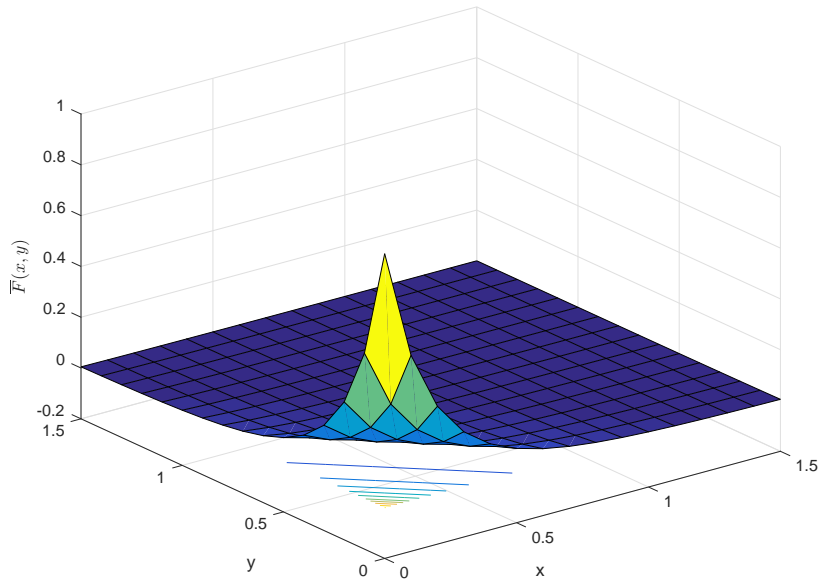}
\caption{De-cumulative distribution function (\ref{ddf-Arnold}).}
\label{Fig-Arnold}
\end{figure}
we depict this joint d.d.f.\, when $\mu_X=\mu_Y=0$, $\sigma_X=\sigma_Y=1$, and $\delta \approx 5.87$. Consequently, $\rho[X, Y] \approx 0.17$. The distribution is symmetric, because $X$ and $Y$ are identically distributed, and quite light-tailed, because $\delta\approx 5.87$ is a fairly large tail-index value.

The pair $(X, Y)$ with joint d.d.f.\, (\ref{ddf-Arnold}) enjoys a number of
attractive properties, which have arguably been a reason for the
popularity of this distribution. Among the properties is the stochastic representation
\begin{equation}
\label{Stoch-Arnold}
(X,Y)=_d \bigg ( \mu_X+\sigma_X \frac{E_X}{G}
,\mu_Y+\sigma_Y \frac{E_Y}{G} \bigg ),
\end{equation}
where ``$=_d$'' denotes equality in distribution, $E_X$ and $E_Y$ are independent exponential r.v.'s with unit rates, and
$G$ is an independent gamma-distributed r.v.\, with shape parameter $\delta>0$ and unit scale. Succinctly, we write $E_X\sim Exp(1)$, $E_Y\sim Exp(1)$ and $G\sim Ga(\delta, 1)$.

\begin{note}\rm
The two exponential random variables $E_X$ and $E_Y$ can be viewed as models of idiosyncratic risks, and the third random variable $G$ as the model for a background, or underlying, risk. Naturally, one may need to deal with more than two dimensions and with other distributions than the exponential and gamma, and all this is indeed possible (cf., e.g., Asimit et al., 2016; and references therein).
\end{note}

A very important for us property of this bivariate distribution is that its regression function is linear, that is, the equation $\mathbf{E}[X|\ Y]=\alpha+\beta Y$ holds, and the parameters are
\[
\alpha=\mu_X+\frac{\sigma_X}{\delta}\left(
1-\frac{\mu_Y}{\sigma_Y} \right)
\quad \textnormal{and} \quad
\beta=\frac{\sigma_X}{\delta\sigma_Y}.
\]
Consequently, Theorem \ref{th-el-1} allows us to calculate the weighted Gini correlation coefficient at a stroke:
\begin{equation}
C_w[X, Y]
= \frac{\sigma_X}{\delta\sigma_Y} \frac{\mathbf{Cov}[Y, w(1-F_Y(Y))]}{\mathbf{Cov}[X, w(1-F_X(X))]}
=\frac{1}{\delta},
\label{par-cw}
\end{equation}
where the right-most equation holds because the ratio of the covariances is equal to $\sigma_Y/\sigma_X$, which is a simple consequence of the fact that $(X-\mu_X)/\sigma_X$ and $(Y-\mu_Y)/\sigma_Y$ are identically distributed.

\begin{note}\rm
For the covariances in equation (\ref{par-cw}) to be finite, the means $\mathbf{E}[X]$ and $\mathbf{E}[Y]$ have to be finite, and this is true whenever $\delta>1$, which we assume. On the other hand, the Pearson correlation coefficient $\rho[X, Y]$ exists and is equal to $1/\delta $ only when $\delta>2$. Consequently, we can say that the weighted Gini correlation coefficient $C_w[X, Y]$ is an extension of the Pearson correlation coefficient to a wider class of random pairs $(X,Y)$, and in particular to those following d.d.f.\, (\ref{ddf-Arnold}) with the tail-index $\delta \in (1, \infty)$. This extension is useful because tail-index values $\delta \in (1, 2]$ have manifested in numerous real-life data sets: insurance, financial, and those related to income inequality (cf., e.g., Greselin et al., 2014; and references therein).
\end{note}

\section{Non-exchangeable linearly-regressed margins}
\label{sec-Par-SF1}

In some situations, exchangeability might be a drawback because identical tail indices of the margins $X$ and $Y$ might contradict empirical evidence. This suggests that the two idiosyncratic r.v.'s $E_X$ and $E_Y$ in  stochastic representation (\ref{Stoch-Arnold}) might be affected differently by the background r.v.\, $G$. To rectify this situation, we can proceed by introducing an auxiliary `background' random variable $G_Y\sim Ga(\delta_Y,1)$ with parameter $\delta_Y>0$, and require it to be independent of all the other random variables. That is, we let the pair $(X,Y)$ admit the stochastic representation
\begin{equation}
\label{Stoch-SF1}
(X,Y)=_d \bigg ( \mu_X+\sigma_X\frac{E_X}{G},
\mu_Y+\sigma_Y\frac{E_Y}{G_Y+G} \bigg ).
\end{equation}
with location parameters $\mu_X,\mu_Y\in\mathbf{R}$ and scale parameters  $\sigma_X,\sigma_Y>0$. The corresponding d.d.f.\, is (Su and Furman, 2016)
\begin{equation}
\label{ddf-SF1}
\overline{F}_{X,Y}(x,y)
=\left(1+\frac{x-\mu_X}{\sigma_X}+\frac{y-\mu_Y}{\sigma_Y}\right)^{-\delta}
\left(1+\frac{y-\mu_Y}{\sigma_Y}\right)^{-\delta_Y}
\end{equation}
for all $x>\mu_X$ and $ y>\mu_Y$, where $\delta>0$ and $\delta_Y>0$ are tail indices. Obviously, $(X-\mu_X)/\sigma_X$ is Pareto of the 2nd kind with the tail index $\delta$, and $(Y-\mu_Y)/\sigma_Y$ is Pareto of the 2nd kind with the tail index $\delta_Y^\ast=\delta+\delta_Y$. In Figure \ref{Fig-Par-ddf-SF1}
\begin{figure}[h!]
\centering
\includegraphics[width=12cm]{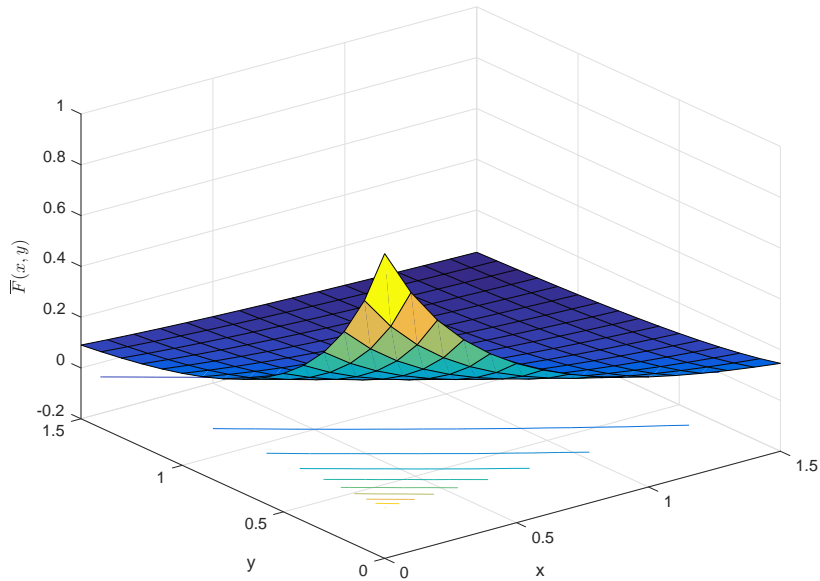}
\caption{De-cumulative distribution function (\ref{ddf-SF1}).}
\label{Fig-Par-ddf-SF1}
\end{figure}
we depict this joint d.d.f.\, when $\mu_X=\mu_Y=0$, $\sigma_X=\sigma_Y=1$, $\delta=2.1$ and $\delta_Y^\ast=2.6254$. As a result of these choices, we have $\rho[X, Y] \approx 0.17$, which is an approximately the same Pearson correlation value as in the previous section. Since the tail indices are different, the joint distribution is not symmetric. Note also that the current distribution is more heavy-tailed than that depicted in Figure \ref{Fig-Arnold}.

Remarkably, despite being an asymmetric distribution, it nevertheless has a linear regression function (Su and Furman, 2016, Theorem 2.3). Indeed, the equation $\mathbf{E}[X|\ Y]=\alpha+\beta Y$ holds with the parameters
\begin{equation}
\alpha=\frac{\sigma_X(\delta_Y^*-1)}{\delta_Y^*(\delta-1)}
\quad \textnormal{and} \quad
\beta
=\frac{\sigma_X(\delta_Y^*-1)}{\sigma_Y\delta_Y^*(\delta-1)}.
\label{sf-parameters}
\end{equation}
Hence, only with a slightly more complex stochastic representation than that in equation (\ref{Stoch-Arnold}), we have succeeded in departing from the symmetry of margins $X$ and $Y$ but preserved the linearity of their regression function. Consequently, we can still enjoy the computational tractability of the weighted Gini correlation coefficient. Namely, Theorem \ref{th-el-1} implies the equation
\begin{equation}
\label{par-fs-00}
C_w[X, Y]=\frac{\sigma_X(\delta_Y^*-1)}{\sigma_Y\delta_Y^*(\delta-1)}
\frac{\mathbf{Cov}[Y,w(1-F_Y(Y))]}{\mathbf{Cov}[X, w(1-F_X(X))]},
\end{equation}
and so we only need to calculate the two covariances on the right-hand side of equation (\ref{par-fs-00}). We note at the outset that, unlike in the previous section, we are now dealing with the case when the standardized covariances do not cancel out, because $(X-\mu_X)/\sigma_X$ and $(Y-\mu_Y)/\sigma_Y$ have different tail indices. Hence, some additional effort is required.

To illustrate, let $w(t)=t^{\gamma }$, in which case $C_w[X, Y]$ reduces to the extended Gini correlation coefficient $\Gamma_\gamma[X, Y]$. With technicalities relegated to Appendix \ref{appendix}, we have the formula
\begin{equation}
\Gamma_\gamma[X, Y]
={1\over \delta} \times \frac{\delta(\gamma+1)-1}{\delta_Y^*(\gamma+1)-1}
\label{par-fs-0}
\end{equation}
whenever $\delta>1$ and $\delta_Y>0$. For comparison, the Pearson correlation coefficient is (Su and Furman, 2016, Proposition 2.2 and Corollary 3.2)
\begin{equation}
\label{Pearscon-SF1}
\rho[X,Y]=\sqrt{\frac{\delta-2}{\delta\delta_Y^*
(\delta_Y^*-2)}}
\end{equation}
whenever $\delta>2 $ and $\delta_Y>0$. Thus, in general, $\Gamma_\gamma[X, Y]$ and $\rho[X,Y]$ do not coincide. We see from equations (\ref{par-fs-0}) and (\ref{Pearscon-SF1}) that the extended Gini correlation coefficient is a decreasing function of $\delta$, whereas the Pearson correlation coefficient is not. We also see this phenomenon in Figure \ref{Fig-Gini-Pearson-SF1},
\begin{figure}[h!]
\centering
\includegraphics[width=12cm]{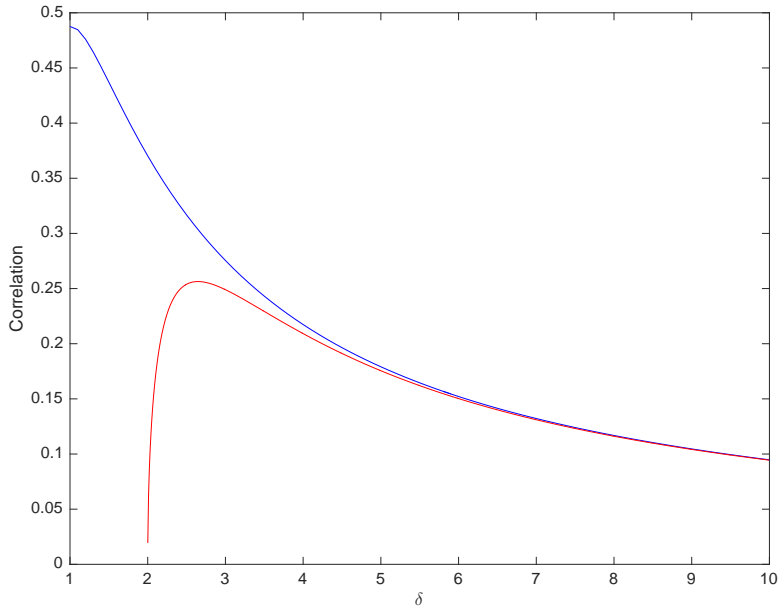}
\caption{Extended Gini (blue curve, decreasing) and Pearson (red curve, non-monotonic) correlation coefficients as functions of $\delta$.}
\label{Fig-Gini-Pearson-SF1}
\end{figure}
where $\mu_X=\mu_Y=0$, $\sigma_X=\sigma_Y=1$, and $\delta_Y=0.5254$.

Our next illustration of Theorem \ref{th-el-1} concerns the possibly $S$-shaped weight function $w(t)=w_{a,b}(t)$, where
\[
w_{a,b}(t)=\frac{1}{B(a,b)}\int_0^t s^{a-1}(1-s)^{b-1}ds
\]
for all $t\in (0, 1)$, with parameters $a>0$ and $b>0$, whose values determine whether the weight function is convex, concave, or $S$-shaped. Obviously, when $b=1$, then $w_{a,b}(t)$ reduces to the power function $t^{a}$ that leads to the extended Gini correlation coefficient $\Gamma_\gamma[X, Y]$ with $\gamma=a$. (Certainly, the function $w_{a,b}$ is known in Calculus as the regularized incomplete beta function, whereas in Statistics it is known as the beta c.d.f.)

Just like in the (convex or concave) case $w(t)=t^{\gamma }$ considered earlier, we can now equally successfully employ Theorem \ref{th-el-1} and concentrate on calculating the two covariances on the right-hand side of equation (\ref{par-fs-00}) with $w(t)=w_{a,b}(t)$. With the technical details relegated to Appendix \ref{appendix}, we obtain the equation
\begin{equation}
C_{w_{a,b}}[X, Y]
={1\over \delta} \times {\displaystyle
1-\frac{B\left(a +1-1/\delta_Y^*,b\right)}{B(a,b)}
-\frac{b}{a+b}
\over\displaystyle
1-\frac{B\left(a+1-1/\delta,b\right)}{B(a,b)}
- {b\over a+b}
} .
\label{par-fs-2}
\end{equation}
When $b=1$, then the right-hand side of equation (\ref{par-fs-2}) reduces to that of equation (\ref{par-fs-0}), which is natural because $w_{a,b}(t)$ is the power function $t^a$ when $b=1$.

\begin{note}\rm
One may naturally wonder whether formulas like (\ref{par-fs-0})--(\ref{par-fs-2}) are of practical value, given theoretical challenges when deriving them. The answer is definitely in affirmative because parametric statistical inference crucially hinges on such formulas.
\end{note}

\section{Non-exchangeable nonlinearly-regressed margins}
\label{sec-Par-SF2}

Here we discuss yet another useful Pareto-type bivariate distribution. We note at the outset that it does not have a linear regression function and thus Theorem \ref{th-el-1} cannot be applied. Namely, let the pair $(X,Y)$ admit the stochastic representation
\begin{equation}
\label{Stoch-SF2}
(X,Y)=_d \bigg ( \mu_X+\sigma_X\frac{E_X}{G_X+G},
\mu_Y+\sigma_Y\frac{E_Y}{G_Y+G} \bigg ),
\end{equation}
where  $E_X$ and $E_Y$ are two independent exponential r.v.'s  with unit scales, and $G\sim Ga(\delta, 1)$, $G_X\sim Ga(\delta_X,1)$ and
$G_Y\sim Ga(\delta_Y,1)$ are three independent gamma r.v.'s, which are also independent of $E_X$ and $E_Y$. Hence, $\mu_X$ and $\mu_Y\in\mathbf{R}$ are location parameters, $\sigma_X$ and $\sigma_Y>0$ are scale parameters, and
$\delta$, $\delta_X$ and $\delta_Y>0$ are tail indices. The joint d.d.f.\, of the random pair $(X, Y)$ is (Su and Furman, 2016)
\begin{equation}
\label{Par-ddf-SF2}
\overline{F}_{X,Y}(x,y)
=\left(1+\frac{x-\mu_X}{\sigma_X}+\frac{y-\mu_Y}{\sigma_Y}\right)^{-\delta}\left(1+\frac{x-\mu_X}{\sigma_X}\right)^{-\delta_X}
\left(1+\frac{y-\mu_Y}{\sigma_Y}\right)^{-\delta_Y}
\end{equation}
for all $x>\mu_X$ and $y>\mu_Y$. We see from the d.d.f.\, that $(X-\mu_X)/\sigma_X$ is Pareto of the 2nd kind with the tail index $\delta_X^\ast=\delta+\delta_X$, and $(Y-\mu_Y)/\sigma_Y$ is Pareto of the 2nd kind with the tail index $\delta_Y^\ast=\delta+\delta_Y$. In Figure \ref{Fig-Par-ddf-SF2}
\begin{figure}[h!]
\centering
\includegraphics[width=12cm]{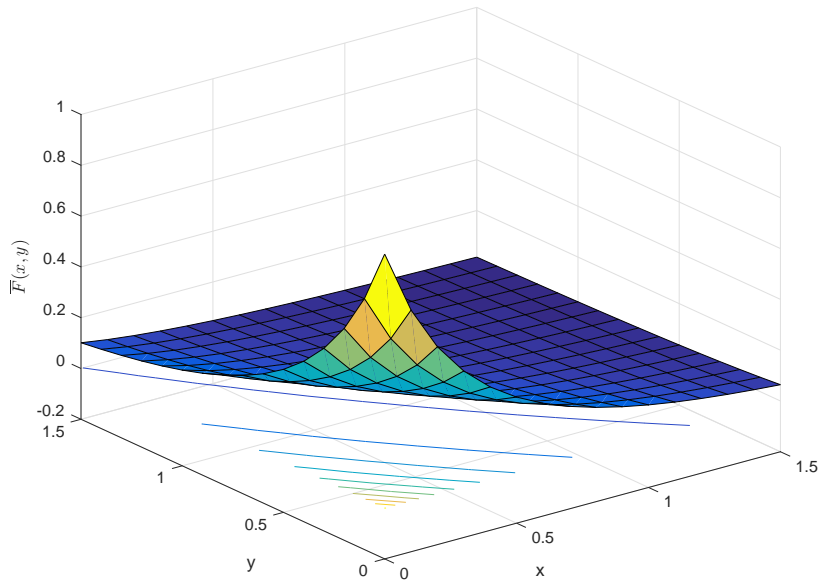}
\caption{De-cumulative distribution function (\ref{Par-ddf-SF2}).}
\label{Fig-Par-ddf-SF2}
\end{figure}
we depict this joint d.d.f.\, when $\mu_X=\mu_Y=0$, $\sigma_X=\sigma_Y=1$, $\delta_X^\ast=3$ and $\delta_Y^\ast=2.5$. Again, we have chosen the parameters so that $\rho[X, Y] \approx 0.17$, which is an approximately the same Pearson correlation value as in the previous sections. The joint distribution is not symmetric due to the different values of $\delta_X^\ast$ and $\delta_Y^\ast$, and it is more heavily-tailed than that depicted in Figure \ref{Fig-Arnold}.

Unlike in the previous sections, the current bivariate Pareto distribution does not have a linear regression function (Su and Furman, 2016, Theorem 3.3), and so we cannot rely on Theorem \ref{th-el-1} to easily calculate Gini-type correlation coefficients. Nevertheless, we can still compute them directly and get closed-form formulas, but the task is considerably more involved.

As an example, consider the extended Gini correlation coefficient  $\Gamma_\gamma[X,Y]$. We need the $(q+1)\times q$ hypergeometric function (Gradshteyn and Ryzhik, 2007)
\begin{eqnarray}
\label{3F2}
_{q+1}F_q(a_1,\ldots,a_{q+1};b_1,\ldots,b_q;z):=\sum_{k=0}^{\infty}\frac{(a_1)_k,\ldots,(a_{q+1})_k }{(b_1)_k,\ldots,(b_q)_k}\frac{z^k}{k!},
\end{eqnarray}
where $q\ge 0$ is a non-negative integer and $(a)_k$ is the Pochhammer symbol. When $a_1,\ldots,a_{q+1}$ are positive, and this is the case upon which we rely, then the radius of convergence of the series is the open disk $|z|<1$ of the complex plane. On the boundary $|z|=1$ of the disc, the series
converges absolutely when $h:=b_1+\cdots + b_q-a_1-\cdots -a_{q+1}>0$, and converges except at $z=1$ when $-1< h \le 0$.

Let $\boldsymbol{i}=(i_1,i_2,i_3)$ be a non-negative integer-valued
triplet such that $i_1+i_2+i_3=2$, and let $\mathcal{I}$ denote the set of all such triplets. We prove (details in Appendix \ref{appendix}) that for $\gamma>0$, $\delta_{X,1,3}^\ast=\delta_X^\ast+i_1+i_3$, and appropriately
defined constants $d_{\boldsymbol{i}}$, the extended Gini correlation coefficient is
\begin{align}
&\Gamma_\gamma[X, Y]
=\frac{\delta_X^\ast(\gamma+1)+1}{\delta_X^\ast\gamma}
\notag
\\
&-\frac{1}{\delta_X^\ast\gamma}
\sum_{ \boldsymbol{i}\in \mathcal{I}}
d_{\boldsymbol{i}}
\frac{(\delta_X^\ast-1)(\gamma+1)(\delta_X^\ast(\gamma+1)-1){}_3F_2(\delta+i_3,2,1;\delta^\ast_{X,1,3},(\gamma+1)\delta_Y^\ast+i_2+i_3;1)}{(\delta^\ast_{X,1,3}-2) (\delta^\ast_{X,1,3}-1)
 ((\gamma+1)\delta_Y^\ast+i_2+i_3-1) }.
\label{gamma-10}
\end{align}
We admit that it is a cumbersome formula, but it is a natural one because it contains, as special cases, earlier derived formulas (\ref{par-cw}) and (\ref{par-fs-0}). To demonstrate, consider first d.d.f.\, (\ref{ddf-Arnold}), in which case we let $\delta_X\downarrow 0$ and $\delta_Y\downarrow 0$, then set
\[
d_{{\boldsymbol{i}}}=\left\{
\begin{array}{ll}
1 & \textrm{when} \quad  {\boldsymbol{i}}=(0,0,2), \\
0 & \textnormal{otherwise},
\end{array}
\right.
\]
and finally use equation (\ref{3F2}) to obtain
\begin{align*}
{}_3F_2(\delta+i_3,2,1;\delta^\ast_{X,1,3},(\gamma+1)\delta_Y^\ast+i_2+i_3;1)
&={}_2F_1(2,1;(\gamma+1)\delta+2;1)
\\
&=\frac{\delta(\gamma+1)+1}{\delta(\gamma+1)-1}.
\end{align*}
With these facts, formula (\ref{par-cw}) follows easily. In a similar fashion, consider d.d.f.\, (\ref{ddf-SF1}), in which case we let $\delta_X\downarrow 0$, then set
\[
d_{{\boldsymbol{i}}}=\left\{
\begin{array}{ll}
1 & \textrm{when} \quad {\boldsymbol{i}}=(0,0,2) \textnormal{ or } (0,1,1),
\\
0 & \textnormal{otherwise},
\end{array}
\right.
\]
and arrive at formula (\ref{par-fs-0}).

\begin{note}\rm
Reflecting upon the bivariate distribution considered in this section, we see that the extended and, more generally, weighted Gini correlations can be calculated in cases that do not have linear regression functions. Indeed, a careful analysis of the proof of Theorem \ref{th-el-1} shows that it is possible to accommodate non-linear regression functions as well, but the right-hand side of equation (\ref{eq-0c}) changes, which is of course natural.
\end{note}

\section*{Acknowledgements}

This work is a part of our research under the grant ``The Capital Asset Pricing Model: an Insurance Variant'' awarded by the Casualty Actuarial Society (CAS).

\appendix
\section{Appendix: proofs}
\label{appendix}

Recall the notation $w^*(t)=1-w(1-t)$, which we use throughout the appendix. 

\begin{proof}[Proof of Theorem \ref{th-1}]
First, we rewrite the weighted correlation coefficient as follows:
\begin{equation}\label{Cor-w-1}
C_w[X,Y]=\frac{\mathbf{Cov}[X,w^*(V)]}{\mathbf{Cov}[X, w^*(U)]},
\end{equation}
where $V=F_Y(Y)$, and $U=F_X(X)$. Using Cuadras's (2002) generalization of Hoeffding's covariance representation, we have
\begin{equation}\label{Cor-w-2}
\mathbf{Cov}[X,w^*(V)]=\iint \Big ( \mathbf{P}[X\le x, V\le t]-F_X(x)t \Big ) dxdw^*(t)
\end{equation}
because $F_V(t)=t$ due to the fact that $V$ is distributed uniformly on the interval $[0,1]$. The upper Fr\'{e}chet bound for the joint probability $\mathbf{P}[X\le x, V\le t]$ is $\min\{F_X(x),t\}$, which is equal to $\mathbf{P}[X\le x, U\le t]$ because $U=F_X(X)$. Thus
\[
\mathbf{P}[X\le x, V\le t]-F_X(x)t \le \mathbf{P}[X\le x, U\le t]-F_X(x)F_U(t)
\]
because $F_U(t)=t$ due to the fact that $U$ is distributed uniformly on the interval $[0,1]$. Consequently, $\mathbf{Cov}[X,w^*(V)]\le \mathbf{Cov}[X, w^*(U)]$, which proves $C_w[X,Y]\le 1$.

To prove $C_w[X,Y]\ge -\lambda_w[X]$, we again start with equation (\ref{Cor-w-2}) but this time use the lower Fr\'{e}chet bound for $\mathbf{P}[X\le x, V\le t]$, which is $\max\{F_X(x)+t-1,0\}$. Note that $\max\{F_X(x)+t-1,0\}-F_X(x)t=-(\min\{F_X(x),1-t\}-F_X(x)(1-t))$. Since $U=F_X(X)$, we thus have
\begin{align*}
\mathbf{P}[X\le x, V\le t]-F_X(x)t
& \ge -\Big ( \mathbf{P}[X\le x, U\le 1-t]-F_X(x)(1-t) \Big )
\\
&=\mathbf{P}[X\le x, 1-U\le t]-F_X(x)t
\\
&=\mathbf{P}[X\le x, 1-U\le t]-F_X(x)F_{1-U}(t) .
\end{align*}
Consequently, $\mathbf{Cov}[X,w^*(V)]\ge \mathbf{Cov}[X, w^*(1-U)]$, which proves the bound
\[
C_w[X,Y]
\ge \frac{\mathbf{Cov}[X,w^*(1-U)]}{\mathbf{Cov}[X, w^*(U)]}
\]
whose right-hand side is equal to $-\lambda_w[X]$. This completes the proof of Theorem \ref{th-1}.
\end{proof}

\begin{proof}[Proof of Theorem \ref{th-2}]
Part (\ref{part-a}) is obvious. To prove part (\ref{part-b}), we rewrite the weighted correlation coefficient as follows:
\begin{equation}\label{Cor-w-1o}
C_w[X,Y]=\frac{\mathbf{Cov}[X,w^*\circ F_Y(Y)]}{\mathbf{Cov}[X, w^*\circ F_X(X)]}. 
\end{equation}
Since the function $w^*\circ F_X$ is non-decreasing, we know from Lehmann (1966) that $\mathbf{Cov}[X, w^*\circ F_X(X)]\ge 0$, and so we only need to show that $\mathbf{Cov}[X,w^*\circ F_Y(Y)]\ge 0$ whenever $X$ and $Y$ are PQD. Using Cuadras's (2002) generalization of Hoeffding's covariance representation, we have
\begin{equation}\label{Cor-w-2o}
\mathbf{Cov}[X,w^*\circ F_Y(Y)]=\iint \Big ( \mathbf{P}[X\le x, Y\le y]-F_X(x)F_Y(y) \Big ) dxdw^*\circ F_Y(y).
\end{equation}
Since the function $w^*\circ F_Y$ is non-decreasing, the integral is non-negative whenever the integrand is non-negative, which is so whenever $X$ and $Y$ are PQD, that is, $\mathbf{P}[X\le x, Y\le y]\ge F_X(x)F_Y(y)$ for all $x$ and $y$. This proves part (\ref{part-b}). The proof of part (\ref{part-c}) is analogous. This concludes the proof of Theorem \ref{th-2}.
\end{proof}

\begin{proof}[Proof of Theorem \ref{th-3}]
We start with the elementary observation that
\begin{equation}\label{wq-1}
C_w[X,Y]=C_w[X_{a,b},Y_{c,d}].
\end{equation}
Next we rewrite the right-hand side of equation (\ref{wq-1}) in the form of a ratio analogous to that on the right-hand side of equation (\ref{Cor-w-1o}). Then, for both the numerator and the denominator, we apply Cuadras's (2002) generalization of Hoeffding's covariance representation, analogously to equation (\ref{Cor-w-2o}). In this way, we express $C_w[X_{a,b},Y_{c,d}]$ in terms of the bivariate cdf of $(X_{a,b},Y_{c,d})$ and the marginal c.d.f.'s of $X_{a,b}$ and $Y_{c,d}$. The interchangeability assumption implies that the latter three c.d.f.'s are equal to those of  $(Y_{c,d},X_{a,b})$, $Y_{c,d}$, and $X_{a,b}$, respectively, thus proving that  $C_w[X_{a,b},Y_{c,d}]$ is equal to $C_w[Y_{c,d},X_{a,b}]$, which is obviously equal to $C_w[Y,X]$. This concludes the proof of Theorem \ref{th-3}.
\end{proof}

\begin{proof}[Proof of Theorem \ref{th-4}]
Since $C_w[a+bX,h(Y)]=C_w[X,h(Y)]$, we only need to show that $C_w[X,h(Y)]=C_w[X,Y]$, which follows from the fact that $F_{h(Y)}(h(y))=F_Y(y)$. This concludes the proof of Theorem \ref{th-4}.
\end{proof}

Proofs of Theorems \ref{th-5} and \ref{th-el-1} rely on the following lemma.

\begin{lemma}\label{lemma-el-1}
When there are constants $\alpha, \beta \in \mathbf{R}$ such that
$\mathbf{E}[X \mid Y]=\alpha+\beta Y$, then the equation
\begin{equation}
\label{eq-0cc}
\mathbf{Cov}[X,w(1-F_Y(Y))]=\beta \mathbf{Cov}[Y, w(1-F_Y(Y))].
\end{equation}
holds for every $w\in \mathcal{W}$.
\end{lemma}

\begin{proof}
We have
\begin{align*}
\mathbf{Cov}[X,w(1-F_Y(Y))]
& =\mathbf{E}[\mathbf{E}[X \mid Y]w(1-F_Y(Y))] -\mathbf{E}[X]\mathbf{E}[w(1-F_Y(Y))]
\\
&=\beta \mathbf{E}[Yw(1-F_Y(Y))] +(\alpha -\mathbf{E}[X])\mathbf{E}[w(1-F_Y(Y))]
\\
&=\beta \mathbf{E}[Yw(1-F_Y(Y))] -\beta \mathbf{E}[Y]\mathbf{E}[w(1-F_Y(Y))]
\\
&=\beta \mathbf{Cov}[Y, w(1-F_Y(Y))],
\end{align*}
which establishes equation (\ref{eq-0cc}).
\end{proof}

\begin{proof}[Proof of Theorem \ref{th-5}] 
We start with the equations 
\begin{align}
\Pi_{G,w}[X, Y]-\mathbf{E}[X]
&=\alpha -\mathbf{E}[X]+\beta \mathbf{E}[Y]
+\beta 
\left(
\pi_{G,w}[Y]-\mathbf{E}[Y]
\right)
\notag 
\\
&= \beta
\left(
\pi_{G,w}[Y]-\mathbf{E}[Y]
\right),
\label{qqq-0}
\end{align}
that follow from the definitions of $\Pi_{G,w}[X, Y]$ and $\pi_{G,w}[Y]$ combined with simple algebra, and the elementary equation $\mathbf{E}[X]=\alpha +\beta \mathbf{E}[Y]$. It now suffices to note that, in view of equation (\ref{eq-0cc}), we have 
\begin{align*}
\beta
&={\mathbf{Cov}[X,w(1-F_Y(Y))]\over  \mathbf{Cov}[Y, w(1-F_Y(Y))]}
\\
&=C_w[X, Y]
\frac{\mathbf{Cov}[X,w(1-F_X(X))]}{\mathbf{Cov}[Y, w(1-F_Y(Y))]}
\end{align*}
which establishes equation (\ref{eq-WIPM-gini}) and concludes the proof of Theorem \ref{th-5}.
\end{proof}

\begin{proof}[Proof of Theorem \ref{th-el-1}]
Equations (\ref{eq-0c}) and (\ref{eq-0cc}) are equivalent, and thus the equation $\mathbf{E}[X|\ Y]=\alpha+\beta Y$ holds. To prove the second (i.e., converse) part of Theorem \ref{th-el-1}, we start with the assumption that equation (\ref{eq-0cc}) holds for every $w\in \mathcal{W}$. With the notation $Z=X-\beta Y$, equation (\ref{eq-0cc}) becomes equivalent to
\begin{equation}
\label{eq-0cd}
\mathbf{E}\big [(Z-\mathbf{E}[Z])w(1-F_Y(Y))\big ]=0,
\end{equation}
which by assumption must hold for all the weight functions $w\in \mathcal{W}$, and in particular for those that safisfy $w(t)>0$ for all $t\in (0,1)$. What remains to be shown, therefore, is that the validity of equation (\ref{eq-0cd}) for all $w\in \mathcal{W}$ implies the validity of 
\begin{equation}
\label{eq-0cde}
Z-\mathbf{E}[Z]=0 \quad \textrm{almost surely}. 
\end{equation}
Indeed, statement (\ref{eq-0cde}) is equivalent to
$X-\beta Y-\mathbf{E}[X]+\beta \mathbf{E}[Y]=0$, and the latter equation obviously  implies $\mathbf{E}[X|\ Y]=\alpha+\beta Y$ with $\alpha =\mathbf{E}[X]-\beta \mathbf{E}[Y]$, thus concluding the proof of the converse of Theorem \ref{th-el-1}.

To verify that statement (\ref{eq-0cde}) holds, we proceed by contradiction, that is, we assume that statement (\ref{eq-0cde}) does not hold. Then there is a subset $A$ of the underlying sample space $\Omega $ such that $\mathbf{P}(A)>0$ and $Z(\omega)-\mathbf{E}[Z] \ne 0 $ for all $\omega \in A$. We can refine this statement by decomposing  $A= A^{-} \cup A^{+}$ with either $A^{-} $ or $ A^{+}$ or both having (strictly) positive probabilities and such that $Z(\omega)-\mathbf{E}[Z]<0 $ for all $\omega \in A^{-}$ and $Z(\omega)-\mathbf{E}[Z]>0 $ for all $\omega \in A^{+}$. Assume for concreteness that   $\mathbf{P}(A^{+})>0$, with the case  $\mathbf{P}(A^{-})>0$ analyzed analogously. We arrive at the required contradiction if $w(1-F_Y(Y(\omega )))>0$ for $\mathbf{P}$-almost all $\omega \in A^{+}$. The latter holds if $F_Y(Y(\omega )) \in (0,1)$ because we deal with the subclass of weight functions such that $w(t)>0$ for all $t\in (0,1)$. Statement $F_Y(Y(\omega )) \in (0,1)$  holds for $\mathbf{P}$-almost all $\omega \in \Omega $ because the c.d.f.\, $F_Y$ is continuous by assumption. This concludes the proof of Theorem \ref{th-el-1}.
\end{proof}

\begin{proof}[Proof of equation (\ref{par-fs-0})]
We need to calculate the two covariances on the right-hand side of equation (\ref{eq-0c}). We start with the denominator and have
\begin{align*}
\mathbf{E}[(X-\mu_X)(1-F_X(X))^\gamma]
&=\sigma_X \int_0^\infty  x(1+x)^{\delta\gamma}\delta(1+x)^{-(\delta+1)}dx
\\
&=\frac{\sigma_X}{(\gamma+1)(\delta(\gamma+1)-1)} .
\end{align*}
Setting $\gamma =0$, we have $\mathbf{E}[X-\mu_X]=\sigma_X/(\delta-1)$. Furthermore, $\mathbf{E}[(1-F_X(X))^\gamma]=1/(\gamma+1)$. From these formulas, we obtain
\begin{equation}
\label{cov-xx}
\mathbf{Cov}[X,(1-F_X(X))^\gamma]
=-{\gamma \over \gamma +1}\sigma_X\frac{\delta}{(\delta-1)
(\delta(\gamma+1)-1)} .
\end{equation}
To calculate the covariance $\mathbf{Cov}[Y, (1-F_Y(Y))^\gamma]$, we recall that $(Y-\mu_Y)/\sigma_Y$ is
Pareto of the 2nd kind with the tail index
 $\delta_Y^\ast=\delta+\delta_Y$. Hence, by changing $\alpha $ into $\alpha+\alpha_Y$ and also $\sigma_X$ into $\sigma_Y$ on the right-hand side of equation (\ref{cov-xx}), we arrive at
\begin{equation}
\label{cov-yy}
\mathbf{Cov}[Y,(1-F_Y(Y))^\gamma]
=-{\gamma \over \gamma +1}\sigma_Y\frac{\delta_Y^\ast}{(\delta_Y^\ast-1)
((\delta_Y^\ast)(\gamma+1)-1)} .
\end{equation}
Using equations (\ref{cov-xx}) and (\ref{cov-yy}) on the right-hand side of equation (\ref{eq-0c}), we obtain equation (\ref{par-fs-0}).
\end{proof}

\begin{proof}[Proof of equation (\ref{par-fs-2})]
With the notation $Z=(X-\mu_X)/\sigma_X$ and the fact that $F_Z^{-1}(u)=(1-u)^{-1/\delta}-1$, we have
\begin{align*}
\mathbf{E}[(X-\mu_X)w_{a,b}(1-F_X(X))]
&=
\sigma_X\mathbf{E}[F_Z^{-1}(F_Z(Z))w_{a,b}(1-F_Z(Z))]
\\
&=\sigma_X\int_0^1 (u^{-1/\delta}-1) w_{a,b}(u) du
\\
&=\sigma_X\bigg (\int_0^1 u^{-1/\delta} w_{a,b}(u) du-\frac{b}{a+b}\bigg ).
\end{align*}
Furthermore,
\begin{align*}
\int_0^1 u^{-1/\delta} w_{a,b}(u) du
&=\int_0^1  u^{-1/\delta}
\int_0^u \frac{t^{a-1}(1-t)^{b-1}}{B(a,b)}dtdu
\\
&=\frac{\delta}{\delta-1}\int_0^1
\frac{t^{a-1}(1-t)^{b-1}}{B(a,b)}
\left(1-t^{(\delta-1)/\delta}\right)dt
\\
&=\frac{\delta}{\delta-1}\left(
1-\frac{B\left(a+(\delta-1)/\delta,b\right)}{B(a,b)} \right).
\end{align*}
Hence, with the formulas $\mathbf{E}[X-\mu_X]=\sigma_X/(\delta-1)$ and $\mathbf{E}[w_{a,b}(1-F_X(X))]=b/(a+b)$, we obtain
\begin{equation}
\label{cov-xx-2}
\mathbf{Cov}[X,w_{a,b}(1-F_X(X))]
=\frac{\sigma_X\delta}{\delta-1}\left(
1-\frac{B\left(a+1-1/\delta,b\right)}{B(a,b)}
- \frac{ b}{a+b}\right).
\end{equation}
(Note that when $b=1$, then equation (\ref{cov-xx-2}) reduces to equation (\ref{cov-xx}).) Changing $\delta $ into $\delta+\delta_Y$ and $\sigma_X$ into $\sigma_Y$ on the right-hand side of equation (\ref{cov-xx-2}), we arrive at
\begin{equation}
\label{cov-yy-2}
\mathbf{Cov}[Y,w_{a,b}(1-F_Y(Y))]
=\frac{\sigma_Y(\delta+\delta_Y)}{\delta+\delta_Y-1}\left(
1-
\frac{B\left(a+1-1/(\delta+\delta_Y),b\right)}{B(a,b)}
-\frac{b}{a+b}\right).
\end{equation}
Using covariance formulas (\ref{cov-xx-2}) and (\ref{cov-yy-2}) on the right-hand side of equation (\ref{par-fs-00}), we obtain equation (\ref{par-fs-2}).
\end{proof}

\begin{proof}[Proof of equation (\ref{gamma-10})]
The bivariate p.d.f.\, of $(X-\mu_X)/\sigma_X$ and $(Y-\mu_Y)/\sigma_Y$ is given by
\begin{equation}\label{pdf-third}
p(x,y;\delta_X,\delta_Y,\delta)
  = \sum_{ \boldsymbol{i}\in \mathcal{I}}
 d_{\boldsymbol{i}}
(1+x)^{-(\delta_X+i_1)}(1+y)^{-(\delta_Y+i_2)}
(1+x+y)^{-(\delta+i_3)},
\end{equation}
where $\mathcal{I}$ is the set of all non-negative and integer-valued triplets $\boldsymbol{i}=(i_1,i_2,i_3)$ such that
$i_1+i_2+i_3=2$. We have
\begin{align*}
{1\over \sigma_X}&\mathbf{E}[(X-\mu_X)(1-F_Y(Y))^\gamma]
\\
&=\sum_{ \boldsymbol{i}\in \mathcal{I}}
d_{\boldsymbol{i}}
 \int_0^\infty \int_0^\infty x(1+y)^{\delta_Y^\ast\gamma} (1+x)^{-(\delta_X+i_1)}(1+y)^{-(\delta_Y+i_2)}
(1+x+y)^{-(\delta+i_3)}dxdy \\
&=\sum_{ \boldsymbol{i}\in \mathcal{I}}
d_{\boldsymbol{i}}
 \int_0^\infty (1+y)^{-((\gamma+1)(\delta_Y+\delta)+i_2+i_3)}  \int_0^\infty x
 (1+x)^{-(\delta_X+i_1)}
\left(1+\frac{x}{1+y}\right)^{-(\delta+i_3)}dxdy.
\end{align*}
We continue with the help of the hypergeometric function noted in Section \ref{sec-Par-SF2} and have
\begin{align*}
{1\over \sigma_X}&\mathbf{E}[(X-\mu_X)(1-F_Y(Y))^\gamma]
\\
&=
\sum_{ \boldsymbol{i}\in \mathcal{I}}
d_{\boldsymbol{i}}
 \int_0^\infty
\frac{(1+y)^{-((\gamma+1)(\delta_Y+\delta)+i_2+i_3)}}{(\delta^\ast_{X,1,3}-2) (\delta^\ast_{X,1,3}-1) }
{~}_2F_1(\delta+i_3,2;\delta^\ast_{X,1,3};y/(1+y))
 dy
 \\
 &=\sum_{ \boldsymbol{i}\in \mathcal{I}}
d_{\boldsymbol{i}} \int_0^1
\frac{ (1-x)^{((\gamma+1)(\delta_Y+\delta)+i_2+i_3-2)}}{(\delta^\ast_{X,1,3}-2) (\delta^\ast_{X,1,3}-1) }
{~}_2F_1(\delta+i_3,2;\delta^\ast_{X,1,3};x)
dx \\
&=
\sum_{ \boldsymbol{i}\in \mathcal{I}}
d_{\boldsymbol{i}}
\frac{{}_3F_2(\delta+i_3,2,1;\delta^\ast_{X,1,3},(\gamma+1)\delta_Y^\ast+i_2+i_3;1)}{(\delta^\ast_{X,1,3}-2) (\delta^\ast_{X,1,3}-1)
 ((\gamma+1)\delta_Y^\ast+i_2+i_3-1) }.
\end{align*}
Since $\mathbf{E}[X-\mu_X]=\sigma_X/(\delta_X^\ast-1)$ and $\mathbf{E}[(1-F_X(X))^\gamma]=1/(\gamma+1)$, equation (\ref{gamma-10}) follows.
\end{proof}

\section*{References}
\def\hang{\hangindent=\parindent\noindent}

\hang
Arnold, B.C. (1983) \textrm{Pareto Distributions.} International Cooperative Publishing House, Fairland, MD.

\hang
Asimit, A.V., Vernic, R.\, and Zitikis, R.\, (2016).
Background risk models and stepwise portfolio construction.
\textrm{Methodology and Computing in Applied Probability} (in press).

\hang
Black, F.\, (1972).
Capital market equilibrium with restricted borrowing,
\textrm{Journal of Business}, \textrm{45}, 444--455.

\hang
Brazauskas, V. and Serfling, R.\, (2003).
Favourable estimators for fitting Pareto models: a study
using goodness-of-fit measures with actual data.
\textrm{ASTIN Bulletin: The Journal of the
International Actuarial Association}, \textrm{33}, 365--381.

\hang
Cuadras, C.M.\, (2002).
On the covariance between functions.
\textit{Journal of Multivariate Analysis}, \textrm{81}, 19--27.

\hang
Fang, K.T., Kotz, S. and Ng, K.W. (1990). \textrm{Symmetric
Multivariate and Related Distributions.} Chapman \& Hall, London.

\hang
Furman, E. and Zitikis, R.\, (2009).
Weighted pricing functionals with application to insurance: An overview.
\textrm{North American Actuarial Journal}, \textrm{13}, 483--496.

\hang
Goovaerts, M. J., Kaas, R., Laeven, R. J. A., Tang, Q. and Vernic, R.\,
(2005). The tail probability
of discounted sums of Pareto-like losses in insurance.
\textrm{Scandinavian Actuarial Journal}, \textrm{6}, 446--461.

\hang
Gradshteyn, I.S. and Ryzhik, I.M. (2007).
\textrm{Table of Integrals, Series and Products}
(7th edition). Academic Press, New York.

\hang
Greselin, F., Pasquazzi, L.\, and Zitikis, R.\, (2014).
Heavy tailed capital incomes: Zenga index,
statistical inference, and ECHP data analysis.
\textrm{Extremes}, \textrm{17}, 127--155.

\hang
Kendall, M. and Stuart, A.\, (1979).
\textrm{The Advanced Theory of Statistics}.
Charles Griffin, London.

\hang
Levy, H. (2011).
\textrm{The Capital Asset Pricing Model in the 21st Century:
Analytical, Empirical, and Behavioral Perspectives.}
Cambridge University Press, Cambridge.

\hang
Lintner, J.\, (1965).
Security prices, risk, and maximal gains from diversification,
\textrm{Journal of Finance}, \textrm{20}, 587--615.

\hang
Samanthi, R.G.M., Wei, W. and Brazauskas, V.\, (2016).
Ordering Gini indexes of multivariate
elliptical risks.
\textrm{Insurance: Mathematics and Economics}, \textrm{68}, 84--91.

\hang
Schechtman, E. and  Yitzhaki, S. (2013).
\textrm{The Gini Methodology: A Primer on a Statistical Methodology.}
Springer, New York.

\hang
Sharpe, W.F.\, (1964).
Capital asset prices: A theory of market equilibrium under
conditions of risk.
\textrm{Journal of Finance}, \textrm{19}, 425--442.

\hang
Shih, W.J. and Huang, W.M. (1992).
Evaluating correlations with proper bounds.
\textrm{Biometrics}, \textrm{48}, 1207--1213.

\hang
Sriboonchita, S., Wong, W.K., Dhompongsa, D.,  Nguyen, H.T.
(2009). \textrm{Stochastic Dominance and Applications to Finance,
Risk and Economics}. Chapman and Hall, Boca Raton, FL.

\hang
Su, J. and Furman, E. (2016). A form of multivariate Pareto
distribution with applications to financial risk measurement.
\textrm{ASTIN Bulletin: The Journal of the
International Actuarial Association} (in press).

\hang
Tversky, A.\, and Kahneman, D. (1992). Advances in prospect
theory: cumulative representation of uncertainty. \textrm{Journal of
Risk and Uncertainty}, \textrm{5}, 297--323.

\hang
Wakker, P.P. (2010). \textrm{Prospect Theory: For Risk and Ambiguity.}
Cambridge University Press, Cambridge.

\end{document}